\documentclass[10pt, psamsfonts]{amsart}
\usepackage{graphicx}
\usepackage[utf8]{inputenc}
\usepackage{latexsym, amsthm, verbatim, amsaddr}
\usepackage{hyperref, comment}
\usepackage{amsfonts, amsmath, amssymb, colonequals}
\usepackage{tikz}
\usetikzlibrary{matrix, arrows, decorations.text}
\usepackage{euscript,mathrsfs, xcolor, enumitem}
\usepackage{lipsum}
\usepackage{geometry}
\newtheorem{theorem}{Theorem}[section]
\newtheorem{lemma}[theorem]{Lemma}
\newtheorem{proposition}[theorem]{Proposition}

\newtheorem{corollary}[theorem]{Corollary}
\numberwithin{equation}{section}
\theoremstyle{definition}
\newtheorem{definition}[theorem]{Definition}
\newtheorem{remark}[theorem]{Remark}

\def\Q{{\mathbb Q}}
\def\GL{{\mathrm{GL}}}			\def\R{{\mathbb R}}
\def\Z{{\mathbb Z}}				
\def\x{{\mathbf x}}				
\def\y{{\mathbf y}}             
\def\h{{\mathbf h}}             \def\f{{\mathbf f}}
\def\a{{\mathbf a}}				\def\e{{\mathbf e}}
\def\y{{\mathbf y}}				\def\N{{\mathbb N}}
\def\C{{\mathbb C}}             \def\q{{\mathbf q}}
\def\bw{{\mathbf w}}            \def\bc{{\mathbf c}}
\def\U{{\mathcal U}}            \def\0{{\mathbf 0}}
\def\A{{\mathcal A}}

\definecolor{cyan(process)}{rgb}{0.0, 0.72, 0.92}

\newcommand{\cov}{\operatorname{cov}}
\newcommand{\K}{\mathcal O_K}

\newcommand{\Kn}{K_S^{m+1}}
\newcommand{\be}{\mathbf{e}}
\newcommand{\vol}{\operatorname{Vol}}
\newcommand{\supp}{\operatorname{supp}}
\newcommand{\rk}{\operatorname{rank}}

\begin{document}

\title{\sc Singular Vectors on Manifolds over totally real Number Fields}
\begin{abstract}  
{\footnotesize  We extend the notion of singular vectors in the context of Diophantine approximation of real numbers with elements of a totally real number field $K$. 
For $m\geq1$, we establish a version of Dani's correspondence in number fields and prove that under a class of `friendly measures' in $K_S^m$, the set of singular vectors has measure zero. Here $S$ is the set of Archimedean valuations of $K$ and $K_S$ is the product of the completions of $\sigma(K)$, $\sigma\in S$. On the other hand, we show the existence of uncountably many non-trivial singular vectors on suitable submanifolds of $K^m_S$ under the action of a certain one parameter subgroup of $\mathrm{SL}_{m+1}(K_S)$.
}
\\
\smallskip

\end{abstract}

\markboth {\hspace*{-9mm} 
\centerline{\footnotesize \sc SINGULAR VECTORS ON MANIFOLDS OVER NUMBER FIELDS} }
{ \centerline{\footnotesize \sc SHREYASI DATTA AND M.M.RADHIKA} \hspace*{-9mm}}

\author{\sc Shreyasi Datta\textsuperscript{*} and \sc M.M.Radhika\textsuperscript{**}}
\thanks{\textsuperscript{*}Department of Mathematics, University of Michigan, Ann Arbor, MI 48109-1043}
\email{\textsuperscript{*}\it dattash@umich.edu}
\thanks{\textsuperscript{**}School of mathematics, Tata Institute of Fundamental Research, Mumbai 400005, India}
\email{\textsuperscript{**}\it mmr@math.tifr.res.in}

\thispagestyle{empty}

\maketitle

\section{Introduction}

 The notion of a singular vector stems from the Dirichlet's theorem which is a fundamental result in the theory of Diophantine approximation. Let us denote the standard inner product in $\R^m$ by $\x\cdot\y$, where $\x,\y\in\R^m.$
 \begin{theorem}[Dirichlet theorem]
 For $\x\in\R^m$, the Dirichlet's theorem states that for any large positive integer $T$, there exist $p\in\Z,\ \mathbf{0}\neq\q\in\Z^m$ such that $\vert\q\cdot\x+p\vert<\frac{1}{T^m}$ with $\Vert\q\Vert\leq T.$
 \end{theorem} 
 The above kind of approximation, often referred to as dual approximation, deals with the closeness of a vector $\x$ to the rational hyperplanes in $\R^m$. Improving Dirichlet's theorem, Khintchine, in 1926 (\cite{KHIN}), introduced the notion of \emph{singular} vectors.
 \begin{definition}[Singular vectors]
 A vector $\x\in\R^m$ is said to be \textit{singular} if for any $0<c$ and for sufficiently large $T$, we can always find nonzero integral points $p\in\Z,\ \mathbf{0}\neq\q\in\Z^m$ satisfying $$\vert\q\cdot\x+p\vert<\frac{c}{T^m},\, \Vert\q\Vert\leq T.$$\end{definition}
  Khintchine showed that a real number is singular if and only if it is rational. He also observed that singular vectors constitute a Lebesgue measure zero set in $\R^m$ (see \cite{Ca} Ch V, §7). 
  
  It is natural to consider Diophantine approximation with rationals coming from a fixed number field and this has been an active area of research recently. There has been extensive work on Diophantine approximation by Gaussian rationals (see \cite{Dani1, Dani2,Dani3, Mum}) but one of the first general results we are aware of are the papers \cite{W1} and \cite{S} of W. M. Schmidt. In addition to the aforementioned papers, analogues of Dirichlet's theorem were proved in \cite{Bu, Ro, S, S2, Hat}. In \cite{EGL}, Einsiedler, Ghosh and Lytle considered vectors which are \emph{badly approximable} by rationals from a number field, and this setting was further investigated in \cite{KL} and in \cite{AGGL}. Following these works, one might wonder about vectors at the opposite end of the spectrum, namely singular vectors with regard to approximation by rationals from a vector field. In this paper we initiate their study and prove the following main results:
  \begin{itemize}
  \item[1)] Paucity of singular vectors: the class of friendly measures introduced in \cite{KLW} gives zero measure to singular vectors in $K^m_S$, see Theorem \ref{fWF} and its corollary for the precise formulation.
\item[2)] Abundance of singular vectors: there are uncountably many totally irrational singular vectors on analytic submanifolds of $K^m_S$ not contained in hyperplanes, see Theorem \ref{manifold} for the precise formulation.
  \end{itemize}
  
  \subsection{Motivation of the work}
  Investigations pertaining to inheritance of Diophantine properties began with a conjecture of K. Mahler's in 1932 (cf. \cite{YBuBook}), which was settled by V. Sprind\v zuk in the 1960s (\cite{Sp2}). Diophantine approximation on manifolds, in particular, the study of singular vectors on manifolds, has been a major topic of research for many decades starting with work of Schmidt and Davenport \cite{DS}, followed by other works \cite{Ba1,Ba2, DRV2, YBu}.
  D. Kleinbock and B. Weiss, in \cite{KW} and \cite{KNW}, studied singular vectors with weights. In \cite{KW}, it was shown that under the class of \textit{friendly measures} introduced in \cite{KLW}, which includes the Lebesgue measure, the set of weighted singular vectors in $\R^n$ has measure zero. The proof therein relies on the quantitative nondivergence results from \cite {KM} and \cite{KLW}. In the study of approximations in $K_S^m$ by vectors from the number field $K$, the following questions arise naturally.
  \begin{itemize}
      \item[1.] What is the correct definition of the class of friendly measures in $K_S^m$?
      
      \item[2.] Does the set of singular vectors in $K_S^m$ have friendly measure zero?
  \end{itemize}

  Khintchine showed in \cite{KHIN} that, unlike in $\R$, there are singular vectors in $\R^2$ that do not belong to any rational affine hyperplane. A vector in $\R^m$ is called \textit{totally irrational} when it does not lie in any rational affine hyperplane. In the study of singular vectors on manifolds, it was shown in \cite{KNW} that there are uncountably many totally irrational singular vectors on a real analytic submanifold $M$ of $\R^m$ if $M$ is not contained inside a rational affine hyperplane. This observation motivates to address the following. 
   \begin{itemize}
   \item[3.] Do all singular vectors in $K_S$ lie in the diagonal embedding of $K$ in $K_S$?
   \item[4.] What is the analogous notion for a totally irrational vector in $K_S^m$?
   \item[5.] What can we say about the set of totally irrational singular vectors in $K_S^m$, or in an analytic submanifold $M=\prod_{\sigma\in S} M_{\sigma}$ of $K_S^m$? Is it uncountable?
   \end{itemize}

  Another reason to  guess an affirmative answer to question 3 above, comes from the relation of the singular vectors with divergent trajectories in the space of unimodular lattices. Maybe replace with: In \cite{Da}, S. G. Dani proves that, for a connected linear semisimple group $G$ and an irreducible non-uniform lattice $\Gamma\subset G$, in the homogeneous space $G/\Gamma$ of the unimodular lattices with $\Q$-rank 1, every divergent trajectory $g_tg\Gamma$ ($g\in G$), under the action by a one-parameter subgroup $\{g_t\}\subset G$, is degenerate divergent. This means, there exists a nonzero vector $v$ in the representation space $\R^n$ such that $g_tgv\to 0$ as $t\to\infty$. In $\R$, this implies that the only singular vectors are the rationals. It is thus natural to expect the same result for $SL_2(K_S)$ which is of $K$-rank $1$.
  
  Section \S\ref{DAN} comprises the definition of a singular vector in $K_S^m$. Here, we have also laid down the statements of the main results in this paper that addresses the question numbers 2, 3 and 5 above. Section \S \ref{DTP} deals with the Dirichlet's approximation theorem for $K_S^m$. Section \S \ref{sec2} cover the Mahler's compactness criterion for the homogeneous space $\mathrm{SL}_{m+1}(K_S)/\mathrm{SL}_{m+1}(\mathcal{O}_K)$, and the Dani correspondence of the divergent trajectories in this space with the singular vectors in $K_S^m$. In Section \S \ref{sec3} we introduce a general class of friendly measures in $K_S^m$ (see Subsection \ref{sec3.1}) and prove that the friendly measure of the set of singular vectors (ref. Def. \ref{DEFSING}) in $K_S^m$ is zero. This answers questions 1 and 2 above. In Section \S \ref{sec4} we provide the proofs for the remaining main theorems stated in Section \S \ref{DAN} that answers the question 3, 4 and 5 above.
  
 \subsection{A preliminary example of the results}
 We explain a simple case of the Theorem \ref{push_main} for the quadratic extension $K=\Q(\sqrt{2})$. The Galois embedding of $\Q(\sqrt 2)$ in $\R^2$ sends $\Q(\sqrt 2) \to\Q(\sqrt 2)\times \Q(\sqrt 2)$ via $$(a+b\sqrt 2)\mapsto (a+b\sqrt 2,a-b\sqrt 2).$$ By the weak approximation for number fields, the image of this embedding is a dense subset of $\R^2$, thus allowing us to consider approximations similar to the Diophantine approximation. We can approximate vectors in $\R^2$ by the elements $(a+b\sqrt 2,a-b\sqrt 2)\in \Q(\sqrt 2)\times \Q(\sqrt 2)$. We can also  extend this approximation to vectors in $\R^4$ (in general for $\R^{2n}$), where $\R^4$ is viewed as the completion of $\Q(\sqrt 2)^2\times \Q(\sqrt 2)^2$. A simple conclusion from Theorem \ref{push_main} is that, Lebesgue almost every vector $\x\in\R^4$ is not a singular vector under this approximation. More importantly, the theorem implies that $$ 
 \text{ for Lebesgue almost every  } \x\in\R^2 , (\x,\x^2)\in\R^4 \text{ is not singular.}$$  
 
 \noindent However, we show that in $\R^2$ all the singular vectors belong to the image of the embedding of $\Q(\sqrt2)$ in $\R^2$. With Theorem  \ref{manifold}, we see that the situation is not so trivial in higher dimensions, for instance, there are uncountably many totally irrational vectors of the form $(\x,\y,\x\y)\in\R^6$, which are singular.
 
  \subsection{Set-up and notation} Let $K$ be a totally real number field of degree $d$ over $\Q$ and $\K$ be the ring of integers in $K$. Let $S:=\{\sigma_1,\cdots,\sigma_d\}$ be the set of all normalized inequivalent archimedean valuations of $K$. For $\sigma\in S$, $K_\sigma$ will denote the completion of the embedding  $\sigma(K)$ in $\R$. By assumption,  $K_\sigma\cong\R$ for each $\sigma\in S$. Denote by $K_S$ the product space $\prod_{\sigma\in S}K_\sigma$. The diagonal embedding of $K\hookrightarrow K_S$ will be denoted by $\tau$ so that, $\tau(\alpha)=(\sigma_1(\alpha),\cdots,\sigma_d(\alpha))$ for $\alpha\in K$. In the sequel, we will identify $K$ with its image $\tau(K)$ in $K_S$ and drop writing $\tau$. For any positive integer $m$, define the inclusion map $$\boldsymbol\tau:=(\tau_1,\cdots,\tau_m):K^m\hookrightarrow K_S^m,$$ where $\tau_i=\tau:K\longrightarrow K_S$ is the diagonal embedding of $K$ into the $i$-th component of $K_S^m$ for $1\leq i\leq m$. As with $K$ in $K_S$, we identify $K^m$ with its image $\boldsymbol{\tau}(K^m)$ and drop writing $\boldsymbol{\tau}$. Apart from the above standard notation, we collect here a few conventions used in this article so as to avoid detailing repeatedly in the sequel.
\noindent
\begin{itemize}
    \item  Elements of $K_S$ and $K_\sigma$ will be denoted by all non bold alphabets, with the corresponding Galois embedding raised on the right side. So, $x^\sigma\in K_\sigma$ and $x=(x^\sigma)_{\sigma\in S}\in K_S$. 
    When $K_\sigma$ or $K_S$ is raised to the m-th power, the vectors will be denoted by bold alphabets with the corresponding Galois embedding written on the top right whenever the vector belongs to $K_\sigma^m$. So, $\x^\sigma=(x_1^\sigma,\cdots,x_m^\sigma)$ belongs to $K_\sigma^m$ and $\x=(x_1,\cdots,x_m)\in\prod_{i=1}^mK_S=K_S^m$. Under the isomorphism $K_S^m\simeq K_{\sigma_1}^m\times\cdots\times K_{\sigma_d}^m$, we will denote $\x=(\x^{\sigma})_{\sigma\in S}$. We will be using either notation depending on the context. An element in a general metric space $X$ will be denoted by a non bold alphabet.
    \item The norm in $K_S$ will be the sup norm given by: For $x=(x^\sigma)_{\sigma\in S}\in K_S$ $$\Vert x\Vert\colonequals\underset{\sigma\in S}\max\vert x^\sigma\vert.$$
    The sup norm in $K_S$ is extended to $K_S^m$ as follows: For $\x=(x_1,\cdots,x_m)\in K_S^m$ 
    $$\Vert\x\Vert\colonequals\underset{1\leq j\leq m}\max\Vert x_j\Vert.$$
    The same notation $\Vert\cdot\Vert$ is retained to represent the sup norm in $K_S^m$ as well. However, boldness of the alphabets within the norm will help distinguish them in the context.
    \item 
   Given a polynomial $p=\sum a_iX_i\in K[x_1,\cdots,x_m]$, where each $X_i$ is a monomial in $x_1,\cdots,x_m$, we denote by $p^\sigma$ the polynomial $p^\sigma=\sum a_i^\sigma X_i\in \sigma(K)[x_1,\cdots,x_m]$, for each $\sigma\in S$. We call a set $\mathcal{W}$ to be a $k$ dimensional \textit{affine $K$-subspace of $K_S^m$} if it is of the form $\mathcal{W}=\prod_{\sigma\in S} W_\sigma$ where, for each $\sigma$, $W_\sigma$ is the affine subspace in $K_\sigma^m$ determined by a collection of affine polynomials  $p_1^\sigma, \cdots, p_{m-k}^\sigma\in \sigma(K)[x_1,\cdots,x_m]$ such that, the collection $p_1,\cdots, p_{m-k}\in K[x_1,\cdots,x_m]$ determines a $k$ dimensional affine $K$-subspace $W$ in $K^m$. We will call a $m-1$ dimensional affine $K$-subspace of $K_S^m$, an \textit{affine $K$-hyperplane in $K_S^m$}.
   \item The terminology - \textit{submanifold of $K_S^m$} - will refer to a subset $M$ of $K_S^m$ which is a product of submanifolds $M_\sigma$ of $K_\sigma^m,\ \sigma\in S$. As $K_\sigma\simeq\R$ for each $\sigma$, we can talk about a real analytic manifold $M_\sigma$ of $K_\sigma^m$. When $M_\sigma$ is a real analytic manifold for all $\sigma\in S$, we call $M$ a \textit{real analytic submanifold of $K_S^m$}.
  
\end{itemize}

\section{Main theorems}
  \label{DAN}
  One of the fundamental results in the theory of Diophantine approximation is the Dirichlet's approximation theorem. A proof of the Dirichlet's theorem for approximation in $K_S$ can be found in \cite{Bu}, \cite{Ro}, \cite{S}, \cite{S2} and \cite{Hat}. 
 We prove the Dirichlet's approximation theorem for $K_S^m$ in  Section \S \ref{DTP}. 
 \begin{theorem}[Dirichlet's theorem]
 \label{Dirichlet}
   Suppose $K$ is a number field of degree $d$ and $c>0$, depends only on the number field $K$ . Given a vector $\x =(x_1,\cdots,x_m)\in K_S^m$, for every $Q>>0$ there are integral vectors $q_0\in\K$ and $\q=(q_1,\cdots,q_m)\in\K^m$, not all zero, and such that $\Vert\mathbf \q\cdot\x+q_0\Vert<\frac{c}{Q^m}, \  \Vert\q\Vert\leq Q$. 
   \end{theorem}
   
  From the above theorem we coin the definition of a singular vector in $K_S^m$ likewise. 
 

 \begin{definition}[Singular vectors]
 \label{DEFSING}
 A vector $\x=(x_1,\cdots,x_m)\in K_S^m$ is said to be singular if for every $c>0$, for all sufficiently large $Q>0$ there exists $\mathbf{0}\neq\q\in \K^m$, $q_0\in\K$ satisfying the following system \begin{equation}\label{sing}
 \begin{aligned}&\Vert \q\cdot\x+q_0\Vert<\frac{c}{Q^m},\\&\Vert \q\Vert\leq Q.\end{aligned}
 \end{equation}
 
\end{definition}

 \noindent
 The set of singular vectors in $K_S^m$ will be denoted $Sing_S^m$. 
 
 One of the main theorems in this paper, based on the notions of nondegeneracy in $K_S^m$ and a more general class of friendly measures on $K_S^m$ introduced in Section \S\ref{sec3}, is the following. This theorem is proved in Section \S\ref{sec3}.
 
 \begin{theorem}\label{push_main}
 \label{fWF}
   Suppose $X=\prod_{\sigma\in S} X_\sigma$ is a Besicovitch space and $\mu=\prod_{\sigma\in S}\mu_\sigma$ be a Federer measure on $X$ and let $\f:X\to K_S^m, \f(x)=(\f_\sigma(x^\sigma))_{\sigma\in S}$ be a continuous map such that $(\f,\mu)$ is nonplanar for $\mu$-almost every point of $X$ and for each $\sigma\in S$, $(\f_\sigma,\mu_\sigma)$ is good for $\mu_\sigma$-almost every point of $X_\sigma$. Then $\f_\star\mu(Sing_S^m)=0$, where $\f_\star\mu$ is the pushforward measure.
   \end{theorem}
   \noindent
 The proof relies on the connection of singular vectors in $K_S^m$ and divergent orbits of the left regular action of certain one parameter subgroups og $G$ in the space  $G/\Gamma$, where $G=\mathrm{SL}_{m+1}(K_S)$, $\Gamma=\mathrm{SL}_{m+1}(\K)$. The above-mentioned connection was first observed by Dani in \cite{Da} for $\mathrm{SL_n(\R)/SL_n(\Z)}$. Next, the main tool of this proof is quantitative nondivergence from \cite{KT}. 
 
 In \cite{KLW} several nontrivial examples of measures $\mu$ and maps $\f$ satisfying the conditions of the above theorem can be found, in particular a class of measures called \textit{friendly} was introduced. Volume measures, fractal measures, pushforward of the previously mentioned two measures by nonsingular nondegenerate maps provide examples to friendly measures. As a consequence of the previous theorem we have the following corollary.
 \begin{corollary}\label{WF}
   Suppose $\mu$ is a friendly measure in $K_S^m$, then $\mu(Sing_S^m)=0$.
   \end{corollary}
 Let us mention a less technical corollary of Theorem \ref{fWF}.
   \begin{corollary}
   Suppose $\f=(\f_\sigma):\mathbf{U}=\prod_{\sigma\in S}U_\sigma\to K_S^m$ be a continuous map, where $U_\sigma\subset K_\sigma^{d_\sigma}$ be an open set. Let $\f$ be a nondegenerate map. Then $\f_\star\lambda(\mathbf{U}\cap Sing_S^m)=0$, where $\lambda$ is the Lebesgue measure.
   \end{corollary}
 
 Now instead of improving the constant in the Dirichlet's theorem, one can try to improve the power of $Q$ on the right hand side of the first inequality in the Dirichlet's theorem. This leads to the definition of the \textit{uniform exponent} $\hat\omega(\x)$ of a vector $\x\in K_S^m$. This exponent was first introduced in \cite{BL} by Y. Bugeaud and M. Laurent, in 2005. It is defined to be the supremum of $v$ such that the following system will have a nonzero solution in $\K^{m+1}$ for all large enough $Q>0$,
 \begin{equation}
 \begin{aligned}
     &\Vert \q\cdot\x+q_0\Vert<\frac{1}{Q^v}\\
     &\Vert \q\Vert \leq Q.
     \end{aligned}
 \end{equation}
It is immediate from the Dirichlet's theorem that $\hat\omega(\x)\geq m$ for every $\x\in K_S^m$. Also, it is easy to see that $\hat\omega(\x)>m$ implies $\x\in K_S^m$ is singular.

 This following theorem, proved in Section \S \ref{sec4}, guarantees that singular vectors in $K_S$ come trivially only. . 
 \begin{theorem}
 \label{simple}
  $x\in K_S$ is singular if and only if it belongs to $K$.
  \end{theorem} 
  The trivial case of singularity for $\x\in K_S^m$ is when there exists $q_0\in\K$, $\0\neq\q\in\K^m$ such that $\Vert \q\cdot \x+q_0\Vert=0$, \textit{i.e.}, $\Vert \q^\sigma\cdot\x^\sigma+q_0^\sigma\Vert=0$ for all $\sigma\in S$. This gives us the following definition of a totally irrational vector in $K_S^m$.
  
 \begin{definition}
 \label{DEFTOT}
  A vector $\x=(x_1,\cdots,x_m)\in K_S^m$ is defined to be totally irrational if for any $\0\neq\q\in\K^m,\ q_0\in\K$, there exists some $\sigma\in S$ such that $\Vert\q^\sigma\cdot \x^\sigma+q_0^\sigma\Vert\neq 0$.
  \end{definition}
  
  The final result in this paper is the following inheritance theorem establishing the existence of nondegenerate divergent orbits on nondegenerate manifolds is $K_S^m$. We refer the reader to Section \S \ref{sec4} for a proof of this result.
  
  \begin{theorem}
  \label{manifold}
  Let $M=\prod_{\sigma\in S} M_\sigma$ be a connected real analytic submanifold of $K_S^m$ such that each $M_\sigma$ is not contained in any affine $K$-hyperplane in  $K_\sigma^m$ and $\dim(M_\sigma)\geq 2$. Then there exists uncountably many totally irrational singular vectors $\x$ in $M$ with $\hat{\omega}(\x)=\infty$. 
  \end{theorem}

\noindent
 \begin{remark}
 \leavevmode
 
  \begin{itemize}[leftmargin=1cm]
  
  \item[(1)] Note that defining $A=\prod_{\sigma_i\in S}H_{\sigma_i}\subset K_S^m$, where each $H_{\sigma_i}$ is a $(m-1)$ dimensional affine $K$-subspaces of $K_{\sigma_i}^m$ defined by a polynomial equation $p_i\in \sigma_i(K)[x_1,\cdots,x_m]$ gives an affine $K$-hyperplane in $K_S^m$. However, it is necessary that the affine $K$-hyperplanes appearing in the proofs of the main theorems in \S\ref{sec4} should have the equation $p_i$ of $H_{\sigma_i}$, be defined by the image under the diagonal embedding of $K$ in $K_S$ of a single polynomial $p$. More precisely, $p_i=p^{\sigma_i}$ for a given $p\in K[x_1,\cdots,x_m]$.
  
  \item[(2)] Although this paper deals only with totally real number fields, the definitions of a singular vector and a totally irrational vector remains the same for a general number field. With the slight modifications necessary in incorporating the complex places, the results here as well as their proofs can be carried out for an arbitrary number field $L$, with $S$ being the set of all archimedean places of $L$.
   \item[(3)] In Theorem \ref{push_main}, the function $\f=(\f_\sigma): X\to K_S^m$ is defined componentwise, \textit{i.e.}, if $X=\prod_{\sigma\in S} K_\sigma^{n_\sigma}$ then, $\f(\x):=(\f_\sigma(\x^\sigma))$ for $\x=(\x^\sigma)\in \prod_{\sigma\in S}K_\sigma^{n_\sigma}$. The functions $\f_\sigma$ are from $K_\sigma^{n_\sigma}\to K_\sigma^{m}$. We could consider the more general situation of $\f(\x):=(\f_\sigma(\x))$ where $\f_\sigma:\prod_{\sigma\in S} K_\sigma^{n_\sigma}\to K_\sigma^m$ but then, the same proof will work only when $\sum n_\sigma \leq m$ and each $\f_\sigma$ is nondegenerate as a map from $\R^{\sum n_\sigma} \to \R^m$. 
   \item[(4)] The techniques in this paper should carry forward to yield similar results for the weighted case as well. 
   
   \item[(5)] Finding Hausdorff dimensions of singular vectors in $\R^m$ has been a very active area of research in the recent times. The papers \cite{C},\cite{CC},\cite{KKLM}, and \cite{DFSU} are a few notable works in this direction. A very  recent paper (\cite{AGMS}) deals with the study of divergent trajectories on a finite product of homogeneous spaces, namely $\prod_i\mathrm{SL}_{m_i+n_i}(\R)/\mathrm{SL}_{m_i+n_i}(\Z)$. This led them to determine the Hausdorff dimension of a certain jointly singular matrix tuples. It would be interesting to find the Hausdorff dimension of the set of singular vectors in $K_S^m$ introduced in this paper.  
  \end{itemize}
  \end{remark}

  \noindent
  \textit{Acknowledgements:} We thank Anish Ghosh for suggesting this problem and for his lectures on dynamics on homogeneous spaces at TIFR. Thanks are due to Anish Ghosh, Ralf Spatzier and Dmitry Kleinbock for giving several useful comments on an earlier draft of this paper. We thank TIFR for providing the research atmosphere necessary for the completion of the work. The first named author thanks University of Michigan for the support given to work on this project. She also thanks Subhajit Jana for his constant support. The second named author is grateful to NBHM for funding the research.

 \section{Dirichlet's theorem for $K_S^m$}
\label{DTP}
 In this section we give the proof of the Dirichlet's theorem stated in the Introduction, in the context of dual approximation in $K_S^m$.
 
 \noindent
 \textit{Notation:} If $M_K$ is the set of all absolute values on $K$, an \textit{$M_K$-divisor} $\mathfrak c: M_K\to\R_+$ is a real valued function on the absolute values satisfying
   \begin{itemize}
       \item $\mathfrak{c}(v)=1$ for all but finitely many $v\in M_K$.
       \item For $v$ a non archimedean valuation, there is an element $\alpha\in K$ such that $\mathfrak{c}(v)=\vert\alpha\vert_v.$
   \end{itemize}
   The \textit{$K$-size} of $\mathfrak{c}$ is defined to be $\Vert\mathfrak c\Vert=\prod_{v\in M_K}\mathfrak{c}(v)^{N_v}$, where $N_v$ is the multiplicity of $v$. We should interpret $\mathfrak{c}(v)$ as prescribing the lengths of the sides of a box (all but a finite number of which are 1) then, the $K$-size of $\mathfrak{c}$ determines the volume of the box. For our purpose we need consider only the $M_K$-divisors $\mathfrak c$ for which $\mathfrak c(v)=1$ for all $v$ non archimedean. Also, $N_v=1$ for all archimedean $v$.

   \begin{proof}[Proof of Theorem \ref{Dirichlet}:]
   Let $X$ be the set of points in $K_S^{m+1}$ given by
   $$X=\left\{(q_1,\cdots,q_m,p)\in K_S^{m+1}\left|~ \begin{aligned}&\Vert q_i\Vert\leq Q{(D_K)}^{\frac{1}{2d}},\ 1\leq i\leq m; \\&
   \Vert \sum_{i=1}^m q_i x_i+p\Vert\leq \frac{{(D_K)}^{\frac{1}{2d}}}{Q^m}\end{aligned}\right.\right\}$$
   where $D_K$ is the discriminant of $K$ and $\sqrt{D_K}$ is the volume of the fundamental domain of the lattice $\K$ in any embedding $K\otimes_\Q\R$. Clearly, $X$ is a convex subset in $K_S^{m+1}$. For each $\sigma_i\in S$, if $v_i$ is the corresponding absolute value, define an $M_K$-divisor $\mathfrak c_1$ with  $\mathfrak{c}_1(v_i)=2Q(D_K)^{\frac{1}{2d}}$. Then, for $1\leq j\leq m$, the contribution to the volume of $X$ from each $K_S$ component equals the $K$-size of $\mathfrak{c}_1$, \textit{i.e.,} $$\Vert\mathfrak c_1\Vert=\prod_{i=1}^d2Q (D_K)^{\frac{1}{2d}}=(2Q)^d\sqrt{D_K}.$$ For $j=m+1$, we define the $M_K$-divisor $\mathfrak c$ to be $\mathfrak c(v_i)=\frac{2}{Q^m}(D_K)^{\frac{1}{2d}}$. An easy computation gives  $\Vert\mathfrak{c}\Vert=\prod_{i=1}^d\frac{2}{Q^m}(D_K)^{\frac{1}{2d}}=\frac{2^d}{Q^{md}}\sqrt{D_K}.$ Now, the volume of $X$  
   $$\vol(X)=\Vert\mathfrak{c}\Vert\prod_{j=1}^m\Vert\mathfrak c_1\Vert=2^{d(m+1)}(\sqrt{D_K})^{m+1}. $$
   Applying Minkowski's convex body theorem, we get the desired lattice point of $\K^{m+1}$. 
   \end{proof}

 \section{Mahler's compactness criterion and Dani Correspondence}
 \label{sec2}

 For an arbitrary integer $m>0$, consider,  $\x=(\x^{\sigma})_{\sigma\in S}\in\prod_{\sigma\in S}K_\sigma^m\simeq K_S^m$. Let $G=\mathrm{SL}_{m+1}(K_S)$ and let $u_\x\in G$ be the matrix $u_{\x}=\begin{bmatrix} 1 &\x\\
 0 & I_{m}\end{bmatrix}$.
 Since $K$ is assumed to be totally real, each $K_\sigma^m\simeq\R^m$ and so $\x^\sigma\in K_\sigma^m$ embeds in $\mathrm{SL}_{m+1}(K_\sigma)$ as 
 $$\x^\sigma\mapsto u_{\x^\sigma}=\begin{bmatrix} 1 &\x^\sigma\\
0& I_{m}\end{bmatrix},$$
 giving an embedding of $K_S^m\hookrightarrow G\simeq\prod_{\sigma\in S} \mathrm{SL}_{m+1}(K_\sigma)$ as $\x\mapsto\prod_{\sigma\in S} u_{\x^\sigma}$ and in $\mathrm{SL}_{d(m+1)}(\R)$ as
 $$\x\mapsto \text{diag}((u_{\x^{\sigma}})_{\sigma\in S})$$
 
 We consider the one parameter subgroup of $G$ denoted by $\{g_t=(g_t^\sigma)\}$ and given by the block diagonal matrices comprising blocks of the one parameter subgroup in each $\mathrm{SL}_{m+1}(K_\sigma)$, that is,
 $$g_{t}^\sigma=\begin{bmatrix} e^{mt} & 0 & \cdots & 0\\0 & e^{-t} & \cdots  & 0\\ \vdots & \vdots & \vdots & \vdots\\0 & 0 & \cdots & e^{-t}
\end{bmatrix}.$$
 
 The discrete subgroup $\Gamma=\mathrm{SL}_{m+1}(\K)$ is a lattice in $G$ and $\{g_t\}$ acts on the homogeneous space $G/\Gamma$ on the left. We identify $G/\Gamma$ with the set $\Omega_{S,m+1}$ of discrete rank $(m+1)$ $\K$-module $g\K^{m+1}$ (corresponding to the elements $g\Gamma$, $g\in G$) in $\Kn$, having an $\K$-basis $\{p_1,p_2,\cdots,p_{m+1}\}$ for $\Kn$ such that, for each $\sigma\in S$, $\{\sigma(p_1),\sigma(p_2),\cdots, \sigma(p_{m+1})\}$ form sides of a parallelopiped of volume 1 in $\R^{m+1}$. 
 
 With the above, we have introduced the notation necessary to establish the Dani correspondence. We further require the $S$-adic version of the  Mahler's compactness criterion for the same. A general version of the Mahler's criterion in the case of a number field, for a finite set of places $S$ containing the archimedean places, is proved in \cite{DKRT} and \cite{DKGT}. We provide below the proof in the case of a totally real number field, for $S=\infty$. As above, let $$\Omega_{S,m+1}=\{g\K^{m+1}\mid g\in G\}.$$
 
 \begin{theorem}[Mahler's Criterion]
 \label{Mah}
A subset $A\subset \Omega_{S,m+1}$ is relatively compact if and only if there exists an open $\varepsilon$-neighbourhood,  $B_\varepsilon$, of $\0$ in $K_S^m$ such that, for all $\Lambda\in A$, $\Lambda\cap B_\varepsilon =\{0\}$.
 \end{theorem}
 
 \begin{proof}
 Let $K[X]\colonequals K[X_{11},\cdots,X_{(m+1)(m+1)}]$ for a deg $d$ number field $K\subset\C$. We identify $M_{m+1}(\C)$ with $\C^{(m+1)^2}$ in the natural way. Consider the algebraic $K$-group $G$ in $\C^{(m+1)^2}$ determined by the zero set of the polynomial $\det(X)-1\in K[X]$. The $K$-points of this group $G(K)$ is $\mathrm{SL}_{m+1}(K)$. Using the Weil restriction of scalars functor $R_{K/\Q}$, we can construct a $\Q$-algebraic subgroup $G'=R_{K/\Q}(G)$ of $\mathrm{SL}_N(\Z)$ (suitably large $N$). The construction is such that, the $\Q$-points of $G'$ are the $K$-points of $G$, \textit{i.e.,} $\mathrm{SL}_{m+1}(K)=G'(\Q)$. We refer the reader to Section \S 2.1.2 in \cite{PLAT} and also Section \S 6.1 in \cite{ZIM} for the construction of $G'$. However, we mention here that if the basis $\{\alpha_1,\cdots,\alpha_d\}$ for $K$ over $\Q$, in the beginning of the construction, is chosen to be an $\K$-basis such that $\K=\sum_{i}\Z\alpha_i$, then we also have $\mathrm{SL}_{m+1}(\K)=G'(\Z)$. Further, as $\sigma_i(K)\subset \R$ for all $i$, we get $\mathrm{SL}_{m+1}(K_S)=G'(\R)$. The Mahler's compactness criterion for $\mathrm{SL}_{m+1}(\R)/\mathrm{SL}_{m+1}(\Z)$ in \cite{Dave} now completes the proof.
 \end{proof}
 
 \begin{theorem}[Dani Correspondence]
 \label{Dani}
   A vector $\x\in K_S^m$ is singular if and only if the corresponding trajectory $\{g_tu_\x\K^{m+1}\mid t\geq 0\}$ is divergent in $G/\Gamma$.
  \end{theorem}
   \begin{proof}
  
  \noindent 
  Suppose $\{g_t u_{\x}\K^{m+1}\}$ is not divergent in $G/\Gamma$. By Theorem \ref{Mah} this means that, there exist an $\varepsilon>0$ and a sequence $t_k\to\infty$ as $k\to\infty$ such that, $\Vert g_{t_k}u_{\x}
  \begin{bmatrix}q^{(k)}_0\\ \vdots\\ q^{(k)}_m\end{bmatrix}\Vert \geq \varepsilon$ for some nonzero $(q^{(k)}_0,\dots,q^{(k)}_m)\in \K^{m+1}$. The inequality above gives
  \begin{equation*}
  \label{duala}
  \Vert\Bigl(e^{mt_k}(q^{(k)}_0+ q^{(k)}_1 x_1+\cdots+ q^{(k)}_m x_m),e^{-t_k} q^{(k)}_1,\cdots,e^{-t_k} q^{(k)}_m \Bigr)\Vert\geq \varepsilon.
  \end{equation*}
  Therefore, for  $Q_k=\varepsilon e^{t_k}$ the system \ref{sing} has no nonzero solution in $\K^{m+1}$ for $c=\varepsilon^{m+1}$. Hence $\x$ is not singular. 
  Conversely, suppose that $\x$ is not singular. Then, there exists an $\varepsilon>0$ such that, for $Q_k\to \infty$ the system \ref{sing} has no nonzero solution in $\K^{m+1}$. Choosing  $t_k>0$ such that $e^{t_k}=Q_k$ yields $g_{t_k} u_{\x}\K^{m+1}$ avoids an open $\frac{\varepsilon}{2}$-neighbourhood of $\mathbf{0}$ as $t_k\to\infty$, i.e., $g_t u_{\x}\K^{m+1}$ does not diverge as $t\to\infty$.

  \end{proof}
  
  \section{Singular vectors and friendly measure}
  \label{sec3}
   \subsection{Friendly measure}\label{sec3.1}
   \subsubsection{Besicovitch space} A metric space $X$ is called \emph{Besicovitch} \cite{KT} if there exists a constant $N_X$ such that the following holds: for any bounded subset $A$ of $X$ and for any family $\mathcal{B}$ of nonempty open balls in $X$ such that
   every  $x \in A$  is a center of some ball in $\mathcal B,$
 there is a finite or countable subfamily $\{B_i\}$ of $\mathcal B$ with
 $$ 1_A \leq \sum_{i}1_{B_i} \leq N_X. $$
 \subsubsection{Federer measure}
 We now define $D$-Federer measures following \cite{KLW}. Let $\mu$ be a Radon measure on $X$, and $U$ an open subset of $X$ with $\mu(U) > 0$. We say that $\mu$ is \textit{$D$-Federer on $U$} if
$$ \sup_{\substack{x \in \supp \mu, r > 0\\ B(x, 3r) \subset U}} \frac{\mu(B(x, 3r))}{\mu(x,r)} < D.$$
We say that $\mu$ as above is \textit{Federer} if, for $\mu$-a.e. $x \in X$, there exists a neighbourhood $U$ of $x$ and $D > 0$ such that $\mu$ is $D$-Federer on $U$. We refer the reader to \cite{KLW} and \cite{KT} for examples of Federer measures.   

\subsubsection{Nonplanar measure}
Suppose $\mu=\prod_{\sigma\in S}\mu_\sigma$ is a measure on $K_S^m$, where $\mu_\sigma$ is a measure on $K_\sigma^m$. We call $\mu$ \textit{nonplanar} if, for all $\sigma\in S$, $\mu_\sigma(\mathcal L_\sigma) = 0$ for any affine hyperplane $\mathcal L_\sigma$ of $K_\sigma^m$. 
Let $X=\prod_{\sigma\in S} X_\sigma$ be a metric space, $\mu=\prod_{\sigma\in S}\mu_\sigma$ a measure on $X$ and $\f=(\f_\sigma)=((f_{1,\sigma}, \cdots, f_{m,\sigma}))$ a map from $X\to K_S^m$. The pair $(\f, \mu)$ will be called \textit{nonplanar at $x_0=(x_0^\sigma)\in X$} if, for any neighborhood $B=\prod_{\sigma\in S}B_\sigma$ of $x_0$, the restrictions of $1, f_{1,\sigma}, \cdots, f_{m,\sigma}$ to $B_\sigma \cap \supp \mu_\sigma$ are linearly independent over $\R$ for all $\sigma\in S$. This is equivalent to saying that for each $\sigma$, $\f_\sigma (B_\sigma\cap\supp\mu_\sigma)$ is not contained in any proper affine subspace of $K_\sigma^m$. Thus, $\mu$ on $K_S^m$ is  nonplanar if and only if $(\mathbf{Id}, \mu)$ is  nonplanar.
\subsubsection{Decaying measure}
Denote by $d_{\mathcal{L}}(\x)$ the (Euclidean) distance from $\x$ to an affine subspace $\mathcal{L}\subset K_\sigma^m$, and let $$\mathcal{L}(\varepsilon)\colonequals \{\x\in K_\sigma^m: d_{\mathcal{L}}(\x)<\varepsilon\}.$$
For $A \subset K_\sigma^m$ with $\mu_\sigma(A)>0$ and $f$ a $K_\sigma$-valued function on $K_\sigma^m$, denote $$\Vert f\Vert_{\mu,A}\colonequals\sup_
{x\in A\cap\,\supp\mu_\sigma}\vert f(\x)\vert.$$
Given $C, \alpha > 0$ and an open subset $U_\sigma$ in  $K_\sigma^m$, we say that $\mu_\sigma$ is \textit{$(C, \alpha)$-decaying on $U_\sigma$} if, for any non-empty open ball $B\subset U_\sigma$ centered in $\supp\mu_\sigma$,
any affine hyperplane $\mathcal L\subset K_\sigma^n$, and any $\varepsilon> 0$ we have $$
\mu_\sigma( B\cap \mathcal{L}(\varepsilon))\leq C {\left(\frac{\varepsilon}{\Vert d_\mathcal{L}\Vert_{\mu_\sigma, B}}\right)}^\alpha \mu_\sigma (B).
$$
\noindent 
Finally, we define \textit{friendly} measures as follows.
\begin{definition}[Friendly measure]
We call a measure $\mu=\prod_\sigma\mu_{\sigma}$ on $K_S^m$ as friendly if,  $\mu$ is nonplanar and $\forall\ \sigma\in S$, for $\mu_\sigma$-a.e. $\x^\sigma\in K_\sigma^m$, there exist a neighborhood $U_\sigma$ of $\x^\sigma$ and positive $C_\sigma,\alpha_\sigma, D > 0$ such that $\mu_\sigma$ is $D$-Federer and $\mu_\sigma$ is $(C_\sigma, \alpha_\sigma)$-decaying on $U_\sigma$.
\end{definition}

\subsection{ Nondegeneracy in $K_S^m$}
\label{nondeg}
 \subsubsection{Good maps}
 Let us recall the notion of $(C,\alpha)$-good functions which was first introduced in \cite{KM}.
 Let $X$ be a metric space and $\mu$ a Borel measure on $X$. For a subset $U$ of $X$ and $C, \alpha > 0$, say that a Borel measurable function $f : U \to \R$ is \textit{$(C, \alpha)$-good on $U$ with respect to $\mu$} if for any open ball $B \subset U$ centered in $\supp \mu$ and $\varepsilon > 0$, we have the condition
 \begin{equation}
 \label{gooddef}
 \mu \left(\{ x \in B : |f(x)| < \varepsilon \} \right) \leq
 C\left(\displaystyle \frac{\varepsilon}{\|f\|_{\mu, B}}\right)^{\alpha}\mu(B).
 \end{equation}
 
 \noindent
 Here, $\|f\|_{\mu, B} = \sup \{c : \mu(\{x \in B : |f(x)| > c\}) > 0\}$.
 Suppose $\f: X\to  \R^n$ is a map and $\mu$ is a Borel measure on $X$, we say $(\f,\mu)$ is \textit{good at $x_0\in X$} if there exists a neighborhood $V$ of $x_0$ and positive $C,\alpha$ such that, any linear combination of $1, f_1,\cdots, f_n$ is $(C,\alpha)$-good on $V$ with respect to $\mu$.\\
 
 \noindent 
  We recall the definition of nondegeneracy as in \cite{KM}. 
  \begin{definition}[Nondegenerate]
Consider a $d$ dimensional submanifold $M = \{\f(\x)\mid \x\in U\}$ of $\R^n$, where $U$ is an open subset of $\R^d$ and $\f = (f_1,\dots,f_n)$ is a $C^m$-imbedding of $U$ into $\R^n$. For $l\le m$, say that $\y = \f(\x)$ is an {\it $l$-nondegenerate point\/} of $M$ if, the space $\R^n$ is spanned by the partial derivatives of $\f$ at $\x$ of order up to $l$.   
\end{definition}
 \noindent We will say that $\y$ is 
 {\it nondegenerate\/} if it is $l$-nondegenerate for some $l$. Finally, we call the manifold $M$ nondegenerate if for $\lambda$-a.e. $\x\in U$, $\f(\x)$ is nondegenerate, where $\lambda$ is the Lebesgue measure on $\R^d$. We can now define nondegenerate maps in $K_S^m$. Suppose $\f=(\f_\sigma):\prod_{\sigma\in S}U_\sigma\rightarrow K_S^m$ be a continuous map, where $U_\sigma\subset K_\sigma^{d_\sigma}$ be an open set. We say \textit{$\f$ is nondegenerate} if each $\f_\sigma:U_\sigma\to K_\sigma^m$ is nondegenerate in the above sense.
 
 \indent Let us recall the following proposition from \cite{KM}, which guarantees that for a  nondegenerate $\f=(\f_\sigma)$ in $K_S^m$, each $\f_\sigma$ is good with respect to the Lebesgue measure $\lambda$ and hence, by Lemma 2.2 in \cite{KT}, $\f$ is $(C, \alpha)$-good for some positive $C, \alpha$. 
 \begin{proposition}
 Let $\f = (f_1,\cdots,f_n)$ be a $C^l$ map from an open subset $U$ of $\R^d$ to $\R^n$, and let $\x_0\in U$ be such that $\R^n$ is spanned by partial derivatives of $\f$ at $\x_0$ of order up to $l$. Then there exists a neighborhood $V\subset U$ of $\x_0$ and a positive $C$  such that any linear combination of $1,f_1,\dots,f_n$ is $(C,1/dl)$-good on $V$.
 \end{proposition}

 \subsection{Quantitative Nondivergence in $K_S^m$}         
 In the following discussion we assume,
	\begin{itemize}
		\item $\mathcal D$ is an integral domain, that is, a  commutative
		ring with
		$1$ and without zero divisors;
		\item $K$ is the quotient field of $\mathcal D$;
		\item  ${\mathcal R}$ is a commutative ring containing
		${K}$ as a subring.
	\end{itemize}	
       If $\Delta$ is a $\mathcal D$-submodule of $\mathcal{R}^m$, we denote by $K\Delta$ (resp. $\mathcal{R}\Delta$) its $K$-(resp.$\mathcal{R}$-) linear span inside $\mathcal{R}^m$, and define the rank of $\Delta$ by
$$\rk(\Delta)\colonequals \dim_K(K\Delta).$$
 \noindent
  We borrow the following notation from \cite{KT} $ \S 6.3$. Denote by $ \GL(m, \mathcal{R})$ the group of $m \times m$ invertible matrices with entries in $\mathcal{R}$ and set
 $$\mathfrak M({\mathcal R},{\mathcal D}, m)\colonequals \{g\Delta~|~g\in \GL(m, \mathcal{R}), \Delta\text{ is a submodule of } \mathcal{D}^m\},$$ and 
 $$\mathfrak P({\mathcal D},m)\colonequals \text{ the set of all nonzero primitive submodules of } \mathcal{D}^m. $$
 A function $\nu : \mathfrak M({\mathcal R},{\mathcal D}, m) \to \mathbb{R}_{+}$ is norm-like if the following three properties hold:
 \begin{enumerate}
 \item[(N1)] For any $\Delta, \Delta' \in \mathfrak M({\mathcal R},{\mathcal D}, m)$ with $\Delta' \subset \Delta$ and $\rk(\Delta') = \rk(\Delta)$ one has $\nu(\Delta') \geq \nu(\Delta)$;
 \item[(N2)] there exists $C_{\nu} > 0$ such that for any $\Delta \in \mathfrak M({\mathcal R},{\mathcal D}, m)$ and any $\gamma \notin \mathcal{R}\Delta$ one has $\nu(\Delta + \mathcal{D}\gamma) \leq C_{\nu} \nu(\Delta)\nu(\mathcal D \gamma)$;
 \item[(N3)] for every submodule $\Delta$ of $\mathcal{D}^m$, the function $\GL(m, \mathcal{R}) \to \mathbb{R}_{+},\ g \to \nu(g\Delta)$, is continuous.
 \end{enumerate}
	
 We will state the following Theorem in $\S6.3$ of \cite{KT} (also Theorem in \S $4.4$ of \cite{DKGT}). 
	
 \begin{theorem}
 \label{QND}
	Let $X$ be a metric space, 		$\mu$ a  uniformly Federer measure  on $X$, and let ${\mathcal D}\subset {K}
		\subset {\mathcal R}$ be as above,
		${\mathcal R}$ being a topological ring. For  $m\in \N$, let  a
		ball $B = B(x_0,r_0)\subset X$ and a continuous map $h:\tilde B \to
		\GL(m,{\mathcal R})$
		be given, where $\tilde B$ stands for $B(x_0,3^mr_0)$. Let $\nu$ be a
		norm-like function on $\mathfrak M({\mathcal R},{\mathcal D},m)$.
		For any $\Delta\in \mathfrak P({\mathcal D},m)$ denote by $\psi_\Delta$ the
		function
		$x\mapsto \nu\big(h(x)\Delta\big)$ on  $\tilde B$.
		Now suppose for some
		$C,\alpha > 0$  one has
		
		{\rm(i)} for every $\Delta\in \mathfrak P({\mathcal D},m)$, the function $\psi_\Delta$
		is $(C,\alpha)$-good on
		$\tilde B$  with respect to
		$\mu$,
		
		{\rm(ii)}  for every $\Delta\in \mathfrak P({\mathcal D},m)$, $\|\psi_\Delta\|_{\mu,B}
		\ge \rho$,
		
		{\rm(iii)} $\forall\,x\in \tilde B\, \cap\, \supp\mu,\quad\#\big\{\Delta\in\mathfrak P({\mathcal D},m)\bigm|
		\psi_\Delta(x) <
		\rho \big\} < \infty. $
		
		\noindent Then 
		for any  positive $\varepsilon\le \rho$ one has
		$$
		\mu\left(\left\{x\in B \left| \nu \big(h(x)\gamma\big) <  \frac{\varepsilon}{C_{\nu}}
		\text{ for \ }\text{some }\gamma\in {\mathcal D}^m\setminus \{0\}
		\right. \right\}\right)\le mC
		\big(N_{X}D_{\mu}^2\big)^m
		\left(\frac\varepsilon\rho \right)^\alpha
		\mu(B).
		$$
	\end{theorem}
	A few notation are in place before stating the next theorem. Define the \textit{content} of a point $\y=(\y^\sigma)\in K_S^m$ as $c(\y)\colonequals\prod_{\sigma\in S} \Vert\y^\sigma\Vert$. Note that $c(\y)\leq \Vert \y\Vert^d$ for any $\y\in K_S^m$. For any $g\in GL(m, K_S)$, we define $\delta(g\K^m)\colonequals\min\{c(\y) ~|~ \y\in g\K^m\setminus \{0\}\}$. Also, we will have  $\cov(g\Delta)$ to be the covolume of the lattice $g\Delta$ in $R=K_S(g\Delta)$ with respect to the normalized volume on $R$. We refer reader to $\S 5$ of \cite{DKGT} for useful results about discrete submodules of $K_S^m$. As a consequence  Theorem \ref{QND} we have the following Theorem in \S 6.3 of \cite{DKGT}. 
   \begin{theorem}
   \label{QND2}
		Let $X$  
		be a Besicovitch metric space,
		$\mu$ a  Federer measure  on $X$, and let
		as above be the set of infinite places of $K$. For  $m\in \N$, let  a
		ball $B = B(x_0,r_0)\subset X$ and a continuous map $h:\tilde B
		\to \GL(m,K_S)$ be given, where $\tilde B$ stands for
		$B(x_0,3^mr_0)$.
		Now, suppose that for some
		$C,\alpha > 0$ and $0 < \rho  < 1$  one has
		
		{\label{rm1}\rm(i)} for every $\,\Delta\in \mathfrak P({\mathcal{O}}_K,m)$, the function
		$\cov\big(h(\cdot)\Delta\big)$ is $(C,\alpha)$-good \ on $\tilde B$  with
		respect to $\mu$;
		
		{\label{rm2}\rm(ii)}  for every $\,\Delta\in \mathfrak P({\mathcal{O}}_K,m)$,
		$\sup\limits_{x\in B\cap\, \supp\mu}\cov\big(h(x)\Delta\big) \ge \rho$.
		
		Then 
		for any  positive $ \varepsilon\le
		\rho$, one has
		$$
		\mu\left(\big\{x\in B\bigm| \delta \big(h(x){\mathcal{O}_K}^m\big) < \varepsilon
		\big\}\right)\le mC \big(N_{X}D_{\mu}^2\big)^m
		\left(\frac{\varepsilon\sqrt{D_K}} {\rho} \right)^\alpha \mu(B).\
		$$
	\end{theorem}

	\subsection{Proof of Theorem \ref{fWF}}
	
   \subsubsection*{Proof of Corollary \ref{WF} assuming Theorem \ref{fWF}}
   Note that $\mu$ is friendly if and only if each $\mu_\sigma$ is Federer, $(\mathbf{Id}, \mu)$ is  nonplanar and $\forall\ \sigma\in S, (\mathbf{Id},\mu_\sigma)$ is good for $\mu_\sigma$-a.e. points. Hence, Corollary \ref{WF} follows from Theorem \ref{fWF}.

   \begin{proposition}\label{prop1}
   	Let $X$ be a Besicovitch metric space and $\mu$ be a Federer measure on $X$. Denote $\tilde B\colonequals B(x,3^{m+1}r)$. Suppose we are given a continuous function $\f : X\to K_S^{m}$ and $C,\alpha >0$ with the following properties:\\
   	{\label{b1}\rm{(i)}}~$x\mapsto \cov(g_t u_{\f(x)}\Delta)$ is $(C,\alpha)$-good with respect to $\mu\text{ in }\tilde B$ for all $\Delta \in \mathfrak P({\mathcal{O}_K},m+1),$\\
   	{\label{b3}\rm{(ii)}} There exists $c>0$ and a sequence $t_i\to\infty$ such that, for any $\Delta \in \mathfrak P({\mathcal O}_K,m+1) $ one has 
   	\begin{equation}
   	\label{A2}
   	\sup_{x\in B\cap\,\supp\mu}\cov (g_{t_i}u_{\f(x)}\Delta)\geq c.
   	\end{equation}
   	Then, $\mu\{x\in B~|~ \f(x) \text{ is singular}\}=0$
   \end{proposition}                     
   \begin{proof}
   	We will check that the map $h=g_tu_{\f}$ satisfies the assumptions of Theorem \ref{QND2} with respect to the measure $\nu=\f_*\mu|_B$ where, condition (\rm{i}) of this proposition is same as the condition (\rm{i}) of Theorem \ref{QND2}. Then condition (\rm{ii}) of this proposition gives (\rm{ii}) of Theorem \ref{QND2}, for all integers $t_i$. Therefore by Theorem \ref{QND2} for any $1>\varepsilon >0$ we have that 
	\begin{align}
	\label{borel_cantelli}
	&\mu\left(\bigg\{x\in B\left | \delta (g_{t_i} u_{\f(x)}\mathcal O_K^{m+1})<\varepsilon c\right.\bigg\}\right)\\
	&\leq (m+1)C(N_X D_{\mu}^2)^{m+1}(\sqrt D_K)^\alpha\varepsilon^\alpha \mu(B)\\
	&= E  \varepsilon^\alpha.
	\end{align}\\
	 From Theorem \ref{Dani} we have $$\mu\{x\in B~|~\f(x) \text{ is singular }\}\subseteq \mu\{x\in B~|~ g_t u_{\f(x)}\Gamma,\ t\geq 0 \text{ is divergent} \}.$$ We want to show that for any $\varepsilon>0$, we have $\mu\{x\in B~|~ \f(x) \text{ is singular}\}\leq \varepsilon^\alpha$. Let $\varepsilon>0$ be given. Then for any $n\in \N$, there exists a set $B_n \subset \{x\in B~|~ \f(x) \text{ is singular}\}$ such that $ \mu ( \{x\in B ~|~ \f(x) \text{ is singular}\})\leq \mu(B_n)+\frac{1}{n}$ and there exists a $t_i$ such that $B_n\subset \bigg\{x\in B\left |\ \delta (g_{t_i} u_{\f(x)}\mathcal O_K^{m+1})<\varepsilon c\right.\bigg\} $. Hence, for every $n\in\N$ and $\varepsilon>0$  we have 
	 $$\mu ( \{x\in B ~|~ \f(x) \text{ is singular}\})\leq E \varepsilon^\alpha+\frac{1}{n},$$ and we conclude.	
   \end{proof}

   \subsubsection*{Proof of theorem \ref{fWF}}
   We will let $\mathcal{R}=K_S$ and $\mathcal{D}=\mathcal{O}_K.$ Let us denote the set of rank $j$  submodules of $\mathcal{D}^{m+1}$ as $\mathcal{S}_{m+1, j}$. For any nonzero submodule $\Delta\in\mathcal{S}_{m+1, j}$, there exists an element $\bw$ of $\bigwedge^j(\mathcal{D}^{m+1})$ such that $\cov(\Delta)\geq (\sqrt {D_K})^j c(\bw)$ and $\cov(g_t u_{\x}\Delta)\geq (\sqrt {D_K})^j c(g_t u_{\x}\bw)$, where $\x\in K_S^m$.  We take $\be_0,\be_1,\cdots,\be_m\in\mathcal{R}^{m+1}$ as the standard basis of $\mathcal R^{m+1}$ over $\mathcal R$ where, $\{\be_i^\sigma\}$ forms the standard basis of $K_\sigma^m$ over $K_\sigma$. Then the standard basis of $\bigwedge^j\mathcal{R}^{m+1}$ will be  $\{\be_I=\be_{i_1}\wedge\cdots\wedge \be_{i_j}~|I\subset\{0,\cdots,m\}\text{ and } i_1<i_2<\cdots<i_j \}$. For an element $\a=\sum a_I \be_I$, we define $\Vert \a\Vert:= \max_{I} \Vert a_I\Vert.$ We can write any $\bw\in\mathcal{D}^{m+1}$ as $\bw=\sum w_I\be_I$, where $w_I\in\mathcal{D}$. 

  Let us note the action of the unipotent flows on the coordinates of $\bw$. Recall, for $\x=(x_1,\cdots,x_m)$ we have $u_{\x}=\begin{bmatrix} 1 &\x\\
  0& I_{m}\end{bmatrix}$ where, the $\sigma$ component is $u_{\x^\sigma}=\begin{bmatrix} 1 &\x^\sigma\\
0& I_{m}\end{bmatrix}$. Also, $g_t$ is taken such that the $\sigma$ component, $g_{t}^\sigma=\begin{bmatrix} e^{mt} & 0 & \cdots & 0\\0 & e^{-t} & \cdots  & 0\\ \vdots & \vdots & \vdots & \vdots\\0 & 0 & \cdots & e^{-t}
\end{bmatrix}$.
Now, observe that  $u_{\x^\sigma}$ leaves $\be_0^\sigma$ invariant and sends $\be_i^\sigma$ to $x^\sigma_i\be_0^\sigma+\be_i^\sigma$ for $i\geq 1$.
Therefore 
\[
u_{\x^\sigma}(\be_I^\sigma)=\left.
\begin{cases}
\be_I^\sigma  &\text{ if } 0\in I\\
\be_I^\sigma +\sum\limits_{i\in I} \pm x^\sigma_i\be_{I\setminus\{i\}\cup\{0\}} & \text{ if } 0\notin I.
\end{cases}
\right.
\]
 
 \noindent  
 Moreover, under the action of $g_t^\sigma$, the vectors $\be_i^\sigma$ are eigenvectors with eigenvalue $e^{-t}$ for $i\geq 1$ and $\be_0^\sigma$ is an eigenvector with eigenvalue $e^{mt}$. Therefore
\[
g_t^\sigma u_{\x^\sigma}(\be_I^\sigma)=\left.
\begin{cases}
e^{(m-j+1)t}\be_I^\sigma  &\text{ if } 0\in I\\
e^{-jt}\be_I^\sigma\pm e^{(m-j+1)t}\sum\limits_{i\in I}  x^\sigma_i\be^\sigma_{I\setminus\{i\}\cup\{0\}} & \text{ if } 0\notin I.
\end{cases}
\right.
\]

Thus, for $\bw\in\bigwedge^j(\mathcal{D}^{m+1}), \bw=\sum w_I\be_I$ with $w_I\in\mathcal{D}$ and so, we get the $\sigma$ component of $g_t u_{\x}\bw$ to be
\begin{equation}
(g_t u_\x\bw)^\sigma= e^{-jt}\sum_{\lbrace I\,|\, 0\notin I\rbrace} w_I\be_I^\sigma+e^{(m-j+1)t}\sum_{\lbrace I\,|\,0\in I\rbrace}\left(w_I+(\sum_{i\notin I}\pm w_{I\setminus\{0\}\cup\{i\}}x^\sigma_i)\right)\be_I^\sigma,
\end{equation}
 
It will be convenient for us to use the following notation, $$\bc(\bw)=\begin{pmatrix}
\bc(\bw)_0\\
\bc(\bw)_1\\
\vdots\\
\bc(\bw)_n
\end{pmatrix},$$
where, $\bc(\bw)_i=\sum\limits_{\substack{J\subset\{1,\cdots,m\} \\\# J=j-1}} w_{J\cup\{i\}}\be_J\in \bigwedge^{j-1}(V_0)$ and $V_0$ is the $\K$ submodule of $K_S^{m+1}$ generated by $\be_1,\cdots,\be_m$. Moreover, $V_0^\sigma$ will be the subspace of $K_\sigma^{m+1}$ generated by $\be_1^\sigma, \cdots, \be_m^\sigma$ so that, $V_0=\prod_{\sigma\in S} V_0^\sigma$.
We may therefore write 
\begin{equation}
(g_t u_{\x}\bw)^\sigma= e^{-jt}\sum_{\lbrace I\,|\,0\notin I\rbrace}w_I\be_I^\sigma+e^{(m-j+1)t}\left(\be_0^\sigma\wedge\sum_{i=0}^m x^\sigma_i\bc(\bw)_i\right)\\
=e^{-jt}\pi_\sigma(\bw)+e^{(m-j+1)t}\be_0^\sigma\wedge\tilde{\x}^\sigma\cdot\bc(\bw),
\end{equation}
where, $\tilde\x=(1,\x)$ and $\pi_\sigma$ is the orthogonal projection from $\bigwedge ^j(K_\sigma^{m+1})\to \bigwedge^j V_0^\sigma$. 
Hence, we have 
$$\begin{aligned}
\cov(g_t u_{\x}\Delta) &\geq (\sqrt{D_K})^m c(g_t{u_\x}\bw)\\&=\prod_{\sigma\in S}\max\bigg( e^{(m-j+1)t}\Vert\sum_{i=0}^m x_i^\sigma \bc(\bw)_i\Vert, e^{-jt} \Vert \pi_\sigma(\bw)\Vert\bigg)\\
\end{aligned}
$$
Thus, for $\x=\f(x)$,
\begin{equation}
\label{covolume}
	\begin{aligned}
	&\sup_{x=(x^\sigma)\in B\cap\, \supp\mu} \cov (g_t u_{\f(x)}\Delta)\\
	&\geq (\sqrt{D_K})^m\prod_{\sigma\in S}\max\left( e^{(m-j+1)t}\sup_{x^\sigma\in B_\sigma\cap\, \supp\mu_\sigma}\Vert \tilde\f_\sigma(x^\sigma)\cdot\bc(\bw) \Vert, e^{-jt}\Vert \pi_\sigma(\bw)\Vert\right),
	\end{aligned}
\end{equation} where, $\tilde{\f}=(1,f_1,\cdots,f_m)$ and $B=\prod B_\sigma$, $B_\sigma$ is a ball in $X_\sigma$.
Since $(\f,\mu)$ is nonplanar, for each $\sigma$, the restrictions of $1,f_{1,\sigma}, \cdots,f_{m,\sigma}$ to $B_\sigma\cap\supp\mu_\sigma$ are linearly independent. Hence, $$\sup_{x^\sigma\in B_\sigma\cap\, \supp\mu_\sigma}\Vert \tilde\f_\sigma(x^\sigma)\cdot\bc(\bw) \Vert\geq \Vert \bc(\bw)\Vert ~\forall~ \sigma\in S,$$ which guarantees condition \ref{A2} of Proposition \ref{prop1}. Since for each $\sigma\in S$, $(\f_\sigma,\mu_\sigma)$ is good for $\mu_\sigma$-almost every point, we can apply Lemma 2.2 from \cite{KT} in order to verify that condition (\rm{i}) of Proposition \ref{prop1} is satisfied. Therefore, by Proposition \ref{prop1} we can conclude the theorem.

  \section{Totally irrational singular vectors on manifolds}
  \label{sec4}
  
  \subsection{Totally irrational vectors}
  From the discussion about totally irrational vectors in Section \S \ref{DAN}, it is clear that the condition for a vector $\x\in K_S^m$ to be not totally irrational is equivalent to saying that $\Vert \q\cdot \x+q_0\Vert= 0$. This in turn implies that $\x$ lies in an affine $K$-hyperplane in $K_S^m$. In $K_S$, this means that all vectors that are not totally irrational are precisely those in $K$. Let us recall the following lemma from \cite{E}.
    \begin{lemma}
  \label{lin}
  There exists an $\varepsilon>0$ depending on the choice of the norm on $\R^2$ such that, no unimodular lattice in $\R^2$ contains two linearly independent vectors each of norm less than $\varepsilon$.
  \end{lemma}

  \subsubsection*{Proof of Theorem \ref{simple} }
  If a vector belongs to $K$ then it is singular is the easy direction. Let $G'=\mathrm{\mathrm{SL}}(2,K_S)$ and $\Gamma'=\mathrm{\mathrm{SL}}(2,\K)$. To prove the other direction, we begin with recalling the identification of $G'/\Gamma'$ with the set $\Omega_{S,2}$ of discrete rank 2 $\K$ modules as described in Section \S \ref{sec2}. For a singular vector $x\in K_S$, by Dani correspondence, the trajectory of lattices $\{g_tu_x\K^2\}$ is divergent. 
  Thus, given $\varepsilon>0$, there exists a $t_\varepsilon>0$ such that 
  \begin{equation}
  \label{DIV}
  \text{for all } t\geq t_\varepsilon, \text{ there exists }\left(\begin{matrix}
  q_t\\
  p_t
  \end{matrix}\right)\in\K^2\text{ with } \Vert g_tu_x\left(\begin{matrix}
  q_t\\
  p_t
  \end{matrix}\right)\Vert<\varepsilon.
  \end{equation}
  
  As each $K_\sigma^2\simeq\R^2$, we can use the above Lemma \ref{lin} to conclude the nonexistence of two linearly independent vectors $\{v_1,\, v_2\}$ in a given lattice $g_tu_x\K^2$ having simultaneously very small norms. The identification in Section \S \ref{sec2} identifies $u_x\Gamma'$ with the $\K$ module $u_x\K^2\in\Omega_{S,2}$ for which there exists a basis $\{v,\, w\}$ such that for each $\sigma\in S$, $v^\sigma,\,w^\sigma$ form the sides of a parallelpiped with area 1 in $K_\sigma^2$.  Let $\varepsilon_\sigma>0$ be the threshold length beyond which any two linearly independent vectors $\{v^\sigma,\, w^\sigma\}$ of a lattice in $\Omega_{S,2}$ cannot coexist inside the $\varepsilon_\sigma$ ball in $K_\sigma^2\simeq\R^2$.  Now, as the number of places is finite, for any $\varepsilon<\min\{\varepsilon_\sigma\mid\,\sigma\in S\}$, continuity of the map $t\overset{\phi}{\mapsto} g_t$ along with the divergence condition in equation (\ref{DIV}) ensures that for any time $t>t_0\geq t_\varepsilon$, the nonzero integral vector $v_t=\left(\begin{matrix}
  p_t\\
  q_t
  \end{matrix}\right)$ which falls into the $\varepsilon$ ball in $K_S^2$ (centered at 0) for the time $v_{t_0}=\left(\begin{matrix}
  p_{t_0}\\
  q_{t_0}
  \end{matrix}\right)$ remains inside the ball, should necessarily be a nonzero integral multiple of $v_{t_0}$ (by Lemma \ref{lin}), \textit{i.e.}, $v_t=\lambda_t v_{t_0}$ for $0\neq\lambda_t\in\K$. Hence, $\Vert g_t u_x \left(\begin{matrix}
  q_{t_0}\\
  p_{t_0}
  \end{matrix}\right)\Vert \to_{t\to\infty} 0\implies e^t\Vert q_{t_0}x+p_{t_0}\Vert=0$ for all large enough $t>0$. This implies that $x$ belongs to $K$.

  \subsection{Inheritance in manifolds over $K_S$}
  \label{MOK}

  \noindent
  \textit{Notation:} Let $H:\,\sum_{i=1}^{m}h_i{\x_i}=h_0$ denote an affine $K$-hyperplane in $K^m$ where, $(h_0,\cdots,h_m)\in\K^{m+1}$ is a primitive lattice point. For the affine $K$-hyperplane $A\colonequals\prod_{\sigma\in S}H_\sigma$ in $K_S^m$, we let $$\vert A\vert\colonequals\Vert(h_1,\cdots,h_m)\Vert$$ be referred to as the norm of $A$.
   Let the collection of affine $K$-hyperplanes of $K_S^m$ as described in Section \S \ref{DAN} be denoted by $\mathcal A$. Note that, this is a countable collection and as a consequence of the Dirichlet's theorem, it is dense in $K_S^m$. \noindent 
   \\

   Let $\Phi\colonequals\K^m\setminus\{0\}\to \R_+$ be a proper function where, $\K^m$ is endowed with the discrete topology. For example, we may consider the norm map. We may extend $\Phi$ to all of $\K^m$ by defining $\Phi(0)=0.$ The irrationality measure function is then defined as $$\eta_{\Phi,\x}(t)\colonequals \min_{\{(q_0,\q)\in\K^{m+1}\setminus\{0\} \mid ~\Phi(\q)\leq t\}}\Vert\q\cdot\x+q_0\Vert.$$
   
   The following abstract result exhibits the existence of totally irrational singular vectors on locally closed subsets of $K_S^m$. This result is a generalization of Theorem 1.1 in \cite{KNW} to the context of number fields. The following theorem follows ideas presented in \cite{KHIN}. We call a subset $M\subset K_S^m$ \textit{locally closed} if $M=\overline{M}\cap W$ for an open set $W$ in $K_S^m$. 
   
  \begin{theorem}
  \label{singular}
  Suppose $M=\prod_{\sigma\in S}M_\sigma$ is a locally closed subset of $K_S^m$. Let $\mathcal L$ and $\mathcal L'$ be a disjoint countable collection of distinct closed subsets of $M$ such that, for each $L_i\in\mathcal L$, $A_i$ is an affine $K$-hyperplane in $K_S^m$ containing $L_i$. Assume the following holds:
  \begin{itemize}
      \item[1)] Collection $\mathcal L$ is dense in $M$. 
      \item[2)] $\mathcal L\cup\mathcal L'=\{\x\in M \mid \x\text{ is not totally irrational}.\}$ 
      \item[3)] For each $L\in\mathcal L$ and $T>0$, $L=\overline{\bigcup\limits_{\substack{L_i\in\mathcal L\\ |A_i|>T}}L\cap L_i}$.
      \item[4)]Transversality condition: For $L\in\mathcal{L}$, $L=\overline{L\setminus(\bigcup\limits_{Y\in \mathcal F}Y\cup\bigcup\limits_{Y'\in\mathcal{F}'}Y')}$, for any finite subcollection $\mathcal F,\ \mathcal{F}'$ of $\mathcal{L},\ \mathcal{L'}$ respectively, such that $L\notin\mathcal F$.
  \end{itemize}
  Then for any non-increasing function $\zeta:\R_+\to\R_+$ and an arbitrary proper function $\Phi:\K^m\setminus\{0\}\to\R_+$, there exists uncountably many totally irrational singular vectors $\x\in M$ with $\eta_{\Phi,\x}(t)\leq\zeta(t)$ for all large $t$.
  \end{theorem}
  \begin{proof}
  Define the set $$\mathcal{S}=\{\x\in M\mid\,\exists\, t_0\text{ such that }\forall\,t\geq t_0,\eta_{\Phi,\x}(t)\leq\zeta(t)\text{ and }\x\text{ is totally irrational} \}.$$
  We produce elements in $\mathcal{S}$ and simultaneously prove $\mathcal{S}$ is uncountable. For this we start with assuming $\mathcal{S}$ is countably infinite, say $\mathcal{S}=\{\x_1,\x_2,\cdots\}$, and arrive at a contradiction as every $\x$ obtained with the method below essentially belongs to $\mathcal S$ but, by (c) below, is not any of the existing elements $\x_i$ in $\mathcal{S}$.
  As $M$ is locally closed, there exists an open subset $W\in K_S^m$ such that $M=\overline{M}\cap W$. Let $\U_0=W$. For the primitive lattice point $(h_0^{(i)},h_1^{(i)},\cdots,h_m^{(i)})$ determining the affine $K$-hyperplane $A_{r_i}$ containing the closed subset $L_{r_i}$ in $M$, we let $p_i\colonequals h_0^{(i)}$ and $\q_i\colonequals(h_1^{(i)},\cdots,h_m^{(i)})$. The idea now is to build a nested sequence $\{\U_i\}_{i=1}^\infty$ of bounded open subsets of $W$, and a sequence of strictly increasing indices $\{r_i\}$ corresponding to the choice of the closed subsets $L_{r_i}$ in each stage $i$. This sequence of sets and indices will satisfy the following conditions:
  \begin{itemize}
      \item[a)] $\overline{\U_i}\subset\U_{i-1}$ for all $i\geq 1$,
      \item[b)] $\Phi(\q_{i})>\Phi(\q_{i-1})$ for all $i\in\N$
      \item[c)] $\U_i$ is disjoint from $L_k\cup L'_{k}\cup\{\x_k\}$ for all $k<i$.
      \item[d)] $\U_i\cap L_{r_i}\neq\emptyset$ for all $i\in\N$.
      \item[e)] for all $i\in\N$ and $\x\in\U_i$, we have the local uniformity condition $\Vert\q_{i-1}\cdot\x +p_{i-1}\Vert<\zeta(\Phi(\q_i))$.
  \end{itemize}
  Proving the existence of this sequence of subsets $\{\U_i\}_{i=1}^\infty$  suffices as any vector $\x$ belonging to the nonempty set $M\cap\bigcap_i\U_i$ is seen to be a  totally irrational singular vector.  With the above notation in place, the construction of the sets $\{\mathcal U_i\}$ and the proof of $\mathcal S$ being uncountable follows, almost verbatim, the proof for Theorem 1.1 in \cite{KNW}. To see $\eta_{\Phi,\x}(t)\leq\zeta(t)$, 
  condition (b) along with $\zeta$ being non-increasing yields the irrationality measure function $\eta_{\Phi,\x}$ to be non-increasing. Let $t_0=\Phi(\q_1)$ then, for any $t\geq t_0$ there exist an $r\in\N$ such that, $t\in[\Phi(\q_r),\Phi(\q_{r+1})]$. Using (e), we get $$\eta_{\Phi,\x}(t)\leq\eta_{\Phi,\x}(\Phi(\q_r))\leq\min\limits_{p\in\K} \Vert \q_r\cdot\x+p\Vert\leq\Vert \q_r\cdot\x+p_r\Vert<\zeta(\Phi(\q_{r+1}))\leq \zeta(t).$$
  
  \end{proof}
  
   In the rest of the article we let $M=\prod_{\sigma\in S} M_\sigma\subset K_S^m$ be a real analytic submanifold of $K_S^m$ (refer Section \S \ref{DAN}). When $\dim(M_\sigma)\geq 2$, in some cases we can reduce the problem of finding totally irrational singular vectors in higher dimensions to surfaces, making it easier to tackle. We have the following proposition towards this end.
   
  \begin{proposition}
  \label{SD2}
  Let $K$ be a totally real number field of degree $d$ over $\Q$. Suppose $M=\prod_{\sigma\in S} M_\sigma$ is a connected real analytic submanifold of $K_S^m$ such that each $M_\sigma$ is not contained in any affine $K$-hyperplane in $K_\sigma^m$ and $\dim(M_\sigma)\geq 2$. Then $M$ contains a bounded real analytic submanifold $N=\prod_{\sigma\in S} N_\sigma$ such that, each $N_\sigma$ is not contained in any affine $K$-hyperplane of $K_\sigma^m$ and $\dim(N_\sigma)=2$ for all $\sigma$.
  \end{proposition}
  
  \begin{proof}
  Let $k_\sigma\colonequals\dim(M_\sigma)$ so that $k=\sum_{\sigma\in S}k_\sigma$ is the dimension of $M$. We prove by inducting on $\dim(M)=k$. If $k_\sigma=2$ for all $\sigma$, the assertion is trivially seen to be true. Assume that the result holds true for all connected real analytic submanifolds $N$ of $K_S^m$ with $2d<\dim(N)<k$. Then, the proof for a connected real analytic manifold $M$ with $\dim(M)=k$ is as follows.
  For each $\sigma$ we identify $K_\sigma^m$ with its image under the isomorphism $K_\sigma^m\simeq\R^m$. We may then assume that each $M_\sigma$ is the image of the open $k_\sigma$-dimensional cube $I_{k_\sigma}\colonequals(0,1)^{k_\sigma}$ under a real analytic immersion $f_\sigma:I_{k_\sigma}\to K_\sigma^m$. Then, $M=\prod_{\sigma\in S} f_\sigma(I_{k_\sigma})$.
  
  Now, consider the connected real analytic submanifold $N$ of $M$ defined as follows. Pick a $\tau\in S$. We then define $N=\prod_{\sigma\in S,\,\sigma\neq\tau}N_\sigma\times N_\tau\colonequals\prod_{\sigma\in S,\,\sigma\neq\tau}(g_\alpha)_\sigma(I_{k_\sigma})\times (g_\alpha)_\tau(I_{k_\tau-1})$ where, the real analytic immersion $g_\alpha$ is given by 
  \begin{align*}
      N_\sigma=(g_{\alpha})_\sigma(I_{k_\sigma})=\,&  f_\sigma(I_{k_\sigma})=M_\sigma, \text{ for all } \sigma\in S, \sigma\neq\tau, \\
      N_\tau=(g_{\alpha})_\tau(x_1,\cdots,x_{k_\tau-1})=\,&  f_\tau(x_1,\cdots,x_{k_\tau-1},\alpha), \mbox{ for some } \alpha\in(0,1).
  \end{align*}
  As $M_\tau$ is not contained in any affine $K$-hyperplane in $K_\tau^m$, it is not contained in $H_\tau$ for any $A=\prod_{\sigma\in S}H_\sigma$ in $\mathcal{A}$. Note that as $A$ varies in $\A$, $H_\tau$ varies over all affine $K$-hyperplanes in $K_\tau^m$. We have the following

  \noindent
  \textit{Claim:} There exists of an $\alpha\in(0,1)$ such that $N_\tau=(g_\alpha)_\tau(I_{k_\tau-1})$ is not contained in $H_\tau$ for any $A=\prod_{\sigma\in S} H_\sigma\in \mathcal A$. 
  
  If the claim holds true, the submanifold $N$ of $K_S^m$ is connected, real analytic and such that $N_\sigma$ is not contained in any affine $K$-hyperplane in $K_\sigma^m$ for all $\sigma$. As $\dim(N)<k$, the induction hypothesis now completes the proof.
  We now establish the claim. Every $A\in\mathcal{A}$ is of the form $A_\h=\prod_{\sigma\in S} H_{\h,\sigma}$ for some primitive lattice point $\h\in\K^{m+1}$ (ref. \ref{MOK}). 
  Let $P=\{\h\in\K^{m+1}\mid\, \h\mbox{ is primitive} \}$. If possible let there exist no $\alpha\in(0,1)$ such that $N_\tau=(g_\alpha)_\tau(I_{k_\tau-1})$ is not contained in any affine $K$-hyperplane in $K_\tau^m$. Thus, for each $\alpha$, there exists an $h\in P$ satisfying $(g_\alpha)_\tau(I_{k_\tau-1})\subset H_{\h,\tau}$. As $\bigcup_{\alpha\in(0,1)}(g_\alpha)_\tau(I_{k_\tau-1})=M_\tau$, taking union over all $\h\in P$ gives
  $$I_{k_\tau}=\bigcup_{h\in P}f_\tau^{-1}(M_\tau\cap H_{\h,\tau}).$$
  Now, as in Lemma 3.5 in \cite{KNW}, given that the right hand side is a countable union, using Baire category theorem we can conclude the existence of an $\h_0\in P$ with the property that $M_\tau\subset H_{\h_0,\tau}$. This gives a contradiction. Hence, the claim .

  \end{proof}
  
  \begin{remark}
  \label{hyperplane}
  In Proposition 3.4 of \cite{KNW}, the hypothesis regards $M$ as a connected real analytic submanifold of $\R^m$ not contained in any proper affine rational subspace. The fact that any proper affine rational subspace is contained in an affine rational hyperplane, is observed in the proof therein. Similar fact holds true in the case of affine $K$-subspaces of $K_S^m$.
  Hence, Proposition \ref{SD2} holds true if affine $K$-hyperplane is replaced by affine $K$-subspace of $K_S^m$.
  \end{remark}
  
  \begin{proposition}
  \label{OBC}
  Let $M=\prod_{\sigma\in S}M_\sigma$ be a bounded connected real analytic submanifold of $K_S^m$ such that each $M_\sigma$ is not contained in any affine $K$-hyperplane in $K_\sigma^m$. Let $\dim(M_\sigma)=2$ for all $\sigma\in S$. Let $A\in\A$ be given by $A=\prod_{\sigma\in S}H_\sigma$. If $F$ is the collection of points $\x\in M\cap A$ such that, for some  $\sigma\in S$, there does not exist a neighbourhood $U_\sigma$ of $\x^\sigma$ with the property that $M_\sigma\cap H_\sigma\cap U_\sigma$ is a real analytic curve then, the connected components of $M\cap A\cap F^c$ are finitely many with their $\sigma$-components being real analytic curves.
  \end{proposition}
  \begin{proof}
  Define for $\sigma\in S$
  $$
      \tilde{F_\sigma}\colonequals \left\{\x^\sigma\in M_\sigma\cap H_\sigma\left|\, \begin{aligned} & \nexists\text{ a neighbourhood } U_\sigma \mbox{ of }\x^\sigma \text{ such that }\\ & M_\sigma\cap H_\sigma\cap U_\sigma \mbox{ is a real analytic curve} \end{aligned}\right.\right \}.$$
 
  Then, $F$ can be written as $F=\bigcup_{\sigma\in S}\Bigl(\tilde{F_\sigma}\times\prod_{\tau\neq\sigma}M_\tau\cap H_\tau\Bigr)$. 
  Now, 
  $$
      F^c=\left\{\x\in M\cap A \mathrel{\bigg|} \begin{aligned} &\exists\text{ a neighbourhood } U=\prod\limits_{\sigma\in S} U_\sigma\mbox{ of }\x\mbox{ such that } \\ & M_\sigma\cap H_\sigma\cap U_\sigma 
      \text{ is a real analytic curve } \forall\,\sigma\in S 
      \end{aligned}\right\}.
  $$
  This implies that $\x\in F^c,\text{ if and only if } \x^\sigma\notin\tilde{F_\sigma}$ for all $\sigma\in S$. We have 
  $$M\cap A\cap F^c=\prod_{\sigma\in S}M_\sigma\cap H_\sigma\cap \tilde{F_\sigma}^c.$$
  Arguing along the lines of Proposition 3.2 of \cite{KNW}, we can conclude that for each $\sigma\in S$, $\tilde{F}_\sigma$ is a finite set and hence $M_\sigma\cap H_\sigma\cap \tilde{F_\sigma}^c$ has only finitely many connected components, each of which is a real analytic curve. 
  \end{proof}
   
  \begin{remark}
  Retaining the terminology from \cite{KW}, a connected component $\gamma$ of $M\cap A\cap F^c,\ A\in\A$, in the above Proposition will be called a \textit{basic connected component} of $M\cap A$.
  Note that $\gamma=\prod_{\sigma\in S}\gamma_\sigma$ where, each $\gamma_\sigma$ is a connected component of $M_\sigma\cap H_\sigma\cap \tilde{F}_\sigma^c$, which is a real analytic curve by the Proposition 3.2 of \cite{KNW}. As $\tilde{F}_\sigma$ is finite for each $\sigma$, the collection $\{\gamma_\sigma\}$ forms a dense subset of $M_\sigma\cap H_\sigma$ and hence, for each $\sigma$, $\{\overline{\gamma_\sigma}\}\supset M_\sigma\cap H_\sigma$. This implies that the set $\{\overline{\gamma}=\prod_{\sigma\in S}\overline{\gamma_\sigma}\}$ comprising the closures of all the basic connected components of $M\cap A,\,A\in\A$ will contain all points of $\{M\cap A\mid\,A\in\A\}$.
  \end{remark}
  
  We recall a few notation before the proof of Theorem \ref{manifold}. For each $\sigma$, $\{\e_1^\sigma,\cdots,\e_m^\sigma\}$ will denote the standard basis of the vector space $K_\sigma^m$ over $K_\sigma$. The bold symbols $\{\e_i\}_{i=1}^m$ will denote the basis of the $K_S$-module, $K_S^m$. Thus, $\e_i=(\e_i^{\sigma})_{\sigma\in S}$. An affine subspace $T=\prod_{\sigma\in S}T_\sigma$ in $K_S^m$ is normal to, say  $\e_1$, will mean that $T_\sigma$ is normal to $\e_1^\sigma$ for all $\sigma\in S$.
  
  \begin{proof}[Proof of Theorem \ref{manifold}:]
  With Remark \ref{hyperplane} in place, we prove the theorem considering affine $K$-hyperplanes of $K_S^m$ in the place of affine $K$-subspaces of $K_S^m$.
  By Proposition \ref{SD2}, we may replace $M$ with a bounded connected submanifold of $K_S^m$ such that for each $\sigma$, $M_\sigma$ is a real analytic surface not contained in any affine $K$-hyperplane in $K_\sigma^m$. For $\x\in M$, the tangent space $T_\x M$ is the product space $\prod_{\sigma\in S}T_{\x^\sigma}M_\sigma$. Let $\tilde{\A}\colonequals\{A_i=\prod_{\sigma\in S}H^{(i)}_\sigma\in\A\mid\, M\cap A_i\neq\emptyset\}$ be the countable subcollection of $\A$ comprising all the affine $K$-hyperplanes of $K_S^m$ that are normal to either $\e_1$ or to $\e_2$. As each $M_\sigma$ is assumed to be not contained in any affine $K$-hyperplane inside $K_\sigma^m$, there exists an $\x\in M$ for which, for all $\sigma\in S$, $T_{\x_\sigma}M_\sigma$ is not normal to either $\e_1^\sigma$ or to $\e_2^\sigma$. For each $\sigma$, we may thus further replace $M_\sigma$ by a smaller connected open subset, possibly an open connected neighbourhood of $\x$ in $M$, so as to ensure that for every $\y\in M$, $T_{\y} M$ is not normal to either $\e_1$ or $\e_2$. Now, for each $\sigma$, $M_\sigma$ can be seen as the graph of a smooth function over its projection to the vector subspace $V_\sigma$ of $K_\sigma^m$ spanned by $\{\e_1^\sigma,\e_2^\sigma\}$. We will choose $M_\sigma$ small enough so that this projection is a convex set for all $\sigma$. 
  Modifying $M$ as above without loss of generality, for any $A_i=\prod_{\sigma\in S}H^{(i)}_\sigma\in\tilde{\A}$ and $\x\in M\cap A$, $T_{\x^\sigma}M_\sigma$ intersects $H^{(i)}_\sigma$ transversally for each $\sigma$, \textit{i.e.}, $T_{\x^\sigma}M_\sigma\cap H^{(i)}_\sigma$ will be an affine one dimensional subspace of $K_\sigma^m$. Using implicit function theorem, we conclude that $M_\sigma\cap H^{(i)}_\sigma$ is a real analytic curve. 
  
  Our aim is to evoke Theorem \ref{singular} to show the existence of totally irrational singular vectors on $M$. For this we first define the collection $\mathcal L\colonequals\{L_i\}$ of closed subsets of $M$. Let $L_i=\prod_{\sigma\in S} L_{i,\sigma}$ where, $L_{i,\sigma}=M_\sigma\cap H^{(i)}_\sigma$ for some $A_i\in\tilde{\A}$. Clearly, each $L_i$ is a product of real analytic curves $L_{i,\sigma}$. Further, each $L_{i,\sigma}$ is connected as the projection of $M_\sigma$ on $V_\sigma$ is chosen to be convex for all $\sigma.$ Thus, $L_i$ is connected for all $i$.
  
  Now, we define the collection $\mathcal L'\colonequals\{L_j'\}$ where $L'_j=\prod_{\sigma\in S} L'_{j,\sigma}$. Each $L'_j$ will be the closure $\overline{\gamma}$ of a basic connected component of a nonempty intersection $M\cap A$, $A\in\A$, such that the $\sigma$-components $L'_{j,\sigma}=\overline{\gamma_\sigma}$ are not contained in any $L_{i,\sigma}$. By Proposition \ref{OBC}, for each $A\in\mathcal{A}$, the basic connected components of $M\cap A$ are finite. Thus, $\{\mathcal{L}'\}$ is a countable collection. It follows from the discussion in the remark above, that $\mathcal L'$ will contain all points of $\{M\cap A\mid\,A\in\A\setminus\tilde{\A}\}$. Hence, $\mathcal{L\cup L'}$ will contain all points of $M$ that are not totally irrational. This establishes (2).
  
  Denote by $l_{i,\sigma}$ the one dimensional affine $K$-subspace in $V_\sigma$ that is the projection of $L_{i,\sigma}$. Then, $\{l_{i,\sigma}\}$, as $L_i$ varies in $\mathcal L$, forms a dense collection of vertical and a dense collection of horizontal lines in $V_\sigma$. Let $V=\prod_{\sigma\in S} V_\sigma$. Clearly, $V$ is a $K_S$-subspace of $K_S^m$ of dimension 2 generated by $\e_1$ and $\e_2$. Thus, $\mathfrak{L}_1\colonequals\{l_i\mid\ A_i\text{ normal to } \e_1\}$ forms a collection of $1$ dimensional affine subspaces of $V$ normal to $\e_1$ and indexed by elements of $K$. Similarly, $\mathfrak{L}_2\colonequals\{l_i\mid\ A_i\text{ normal to } \e_2\}$ is a dense collection of 1 dimensional subspaces in $V$ normal to $\e_2$ and indexed by the elements of $K$. The collections $\mathfrak{L}_1,\text{ and }\mathfrak{L}_2$ are dense in $V$ as $K$ is dense in $K_S$ by weak approximation. Given $x\in K_S$, weak approximation gives a sequence $p_i/q_i\in K$, such that $\underset{i\to\infty}{\lim}p_i/q_i=x$. Thus, all points $\x\in M$ with $T_\x M$ not normal to $\e_1$ and $x_1\in K_S$ being the $\e_1$ coordinate in the projection $V$, can be approximated by a sequence of points $\{\x_i\}\subset L_i$, where $\x_i$ has $p_i/q_i$ as the $\e_1$ coordinate in $V$, lying on a hyperplane $l_i\in\mathfrak{L}_1$ and $\{p_i/q_i\}$ approximate $x_1\in K_S$ via weak approximation. One can argue similarly to approximate a point $\x\in M$ having projection $x_2\in K_S$ lying in the span if $\e_2$ in $V$. This proves (1), the density of $\mathcal L$ in $M$.
  
  To prove (3), note that for $\tilde{\x}=(x_1,x_2)\in l_i$, if $p_i/q_i\rightarrow x_2$ then, we can get a sequence of points $\tilde{\x_i}=(x_1,p_i/q_i)$ lying on $l_i\in\mathfrak{L}_2$. Note that $\vert A_j\vert=\Vert q_j\Vert$ tends to $\infty$ as $\tilde{\x_i}$ tends to $\tilde{\x}$. Thus, for any $T>0$, and $\x\in L_i$ with $\tilde{\x}$ as its projection in $V$, we can find affine $K$-hyperplanes $A_j$ with $\vert A_j\vert>T$, giving a sequence of points lying on $L_i\cap L_j$ converging to $\x$.
  
  For a given $L_i\in\mathcal L$, each $L_{i,\sigma}$ is a real analytic curve. Hence, for $j\neq i$, $L_{i,\sigma}\cap L_{j,\sigma}$ is either empty or consists of a single point. If $\mathcal F=\{L_k\}_{k=1}^n\subset\mathcal L$, with $L_i\neq L_k\forall\ k$ then, the above discussion implies that for each $\sigma$, $L_{i,\sigma}\setminus(\bigcup_k L_{k,\sigma})$ is dense in $L_{i,\sigma}$. Thus, $\prod_{\sigma\in S}\Bigl(L_{i,\sigma}\setminus(\bigcup_k L_{k,\sigma})\Bigr)$ is dense in $L_i$. Now,
  
  \begin{equation}
  \label{eqnL}
  \prod_{\sigma\in S}\Bigl(L_{i,\sigma}\setminus\bigcup_k L_{k,\sigma}\Bigr)\subset L_i\setminus\prod_{\sigma\in S}\Bigl(\bigcup_kL_{k,\sigma}\Bigr)\subset L_i\setminus\bigcup_k L_k    
  \end{equation}
  
  This implies $L_i\setminus(\bigcup_k L_k)$ is dense in $L_i$. Following the arguments used to establish condition (c) in the proof of Theorem 1.7 in \cite{KNW}, we can see that each $L'_{k',\sigma}\cap L_{i,\sigma}$ has no nonempty interior relative to the $\sigma$-adic topology on $L_{i,\sigma}$. Thus, $L_{i,\sigma}\setminus(\bigcup_{k'}L'_{k',\sigma})$ is dense in $L_{i,\sigma}$ for each $\sigma$. Hence using \ref{eqnL} with $L'_{k'}$ in place of $L_k$, we may conclude that $L_i\setminus(\bigcup_{k'}L'_{k'})$ is dense in $L_i$. This proves condition (4) of Theorem \ref{singular}.
  
  To see $\hat{\omega}(\x)=\infty$, notice that, in light of the definition of the irrationality measure function $\eta_{\Phi,\x}$, the uniform exponent of $\x\in K_S^m$ defined in the introduction can be interpreted as the following for $\Phi=\Vert\cdot\Vert$
   $$\hat{\omega}(\x)=\sup \left\{\nu:\limsup\limits_{t\rightarrow\infty}{t^\nu\eta_{\,\Vert\cdot\Vert,\x}(t)}<\infty\right\}.$$
   Considering $\zeta(t)= e^{t}/t^\nu$ in the inequation $\eta_{\,\Vert\cdot\Vert,\x}(t)\leq\zeta(t)$ of Theorem \ref{singular} now yields the desired conclusion.

  \end{proof}

\bibliographystyle{abbrv}
\bibliography{mybibfiles}

\end{document}